\documentclass{amsart}
\usepackage{amsmath}
\usepackage{hyperref}
\usepackage{setspace}

\allowdisplaybreaks
\DeclareSymbolFont{bbold}{U}{bbold}{m}{n}
\DeclareSymbolFontAlphabet{\mathbbold}{bbold}
\newcommand{\ind}{\mathbbold{1}}

\numberwithin{equation}{section}

\theoremstyle{plain} \newtheorem{theorem}{Theorem}[section]
\theoremstyle{plain} 
\theoremstyle{plain} \newtheorem{corollary}[theorem]{Corollary}
\theoremstyle{plain} \newtheorem{proposition}[theorem]{Proposition}
\theoremstyle{remark} \newtheorem{remark}[theorem]{Remark}
\theoremstyle{definition} \newtheorem{example}[theorem]{Example}
\theoremstyle{definition} \newtheorem{definition}[theorem]{Definition}
\theoremstyle{remark} 

\newcommand{ \R}{ \mathbb R }
\newcommand{\E}{ \mathbb E}

\renewcommand{\Pr}{ \mathbb P}

\newcommand{\Qr}{ \mathbb Q}

\newcommand {\cG}{\mathcal{G}}
\newcommand {\cC}{\mathcal{C}}
\newcommand {\cD}{\mathcal{D}}

\newcommand{\bR}{\mathbb R}
\newcommand{\bP}{\mathbb P}

\def\be{\begin{equation}}
\def\ee{\end{equation}}

\begin{document}
\title{Transient one-dimensional diffusions conditioned to converge to a different limit point}

\author[A. Hening]{Alexandru Hening }
\thanks{A.H. was supported by EPSRC grant EP/K034316/1}
\address{Department of Statistics \\
 1 South Parks Road \\
 Oxford OX1 3TG \\
 United Kingdom}
 \email{hening@stats.ox.ac.uk}

\maketitle

\begin{abstract}
Let $(X_t)_{t\geq 0}$ be a regular one-dimensional diffusion that models a biological population. If one assumes that the population goes extinct in finite time it is natural to study the $Q$-process associated to $(X_t)_{t\geq 0}$. This is the process one gets by conditioning $(X_t)_{t\geq 0}$ to survive into the indefinite future.

The motivation for this paper comes from looking at populations that are modeled by diffusions which do not go extinct in finite time but which go `extinct asymptotically' as $t\rightarrow \infty$.
We look at transient one-dimensional diffusions $(X_t)_{t \geq 0}$ with state space $I=(\ell, \infty)$ such that $X_t\rightarrow \ell$ as $t\rightarrow \infty$, $\Pr^x$-almost surely for all $x\in I$. We `condition' $(X_t)_{t \geq 0}$ to go to $\infty$ as $t\rightarrow \infty$ and show that the resulting diffusion is the Doob $h$-transform of $(X_t)_{t\geq 0}$ with $h=s$ where $s$ is the scale function of $(X_t)_{t\geq 0}$. Finally, we explore what this conditioning does in two examples.
\end{abstract}

\section{Introduction}

Let $(X_t)_{t\geq 0}$ be a one dimensional diffusion with state space $I:=(0,\infty)$. If one assumes that $(X_t)_{t\geq 0}$ models a population, then the hitting time of 0,
\[
T_0:=\inf\{t\geq 0: X_t=0\},
\]
can be interpreted as the time of extinction. Throughout the paper we take $(X_t)_{t\geq 0}$ to be the canonical process on $C(I)$ so that our diffusion is described by a family of probability measures \[(\Pr^x)_{x\in I}\]
for which
\[
\Pr^x\{X_0=x\}=1.
\]

In mathematical biology and ecology it is of interest to know what the distribution of the population looks like before extinction. Consequently, if $\Pr^x\{T_0<\infty\}=1$ for all $x\in I$, it is relevant to study the so-called $Q$-process which corresponds to conditioning the process $(X_t)_{t\geq 0}$ to survive into the indefinite future.
\begin{definition}
For any $s\geq 0$ and any Borel set $B\subset C((0,s])$ consider the limit
\[
Q^x\{X\in B\} = \lim_{t\rightarrow \infty} \Pr^x\{X\in B~|~T_0>t\}.
\]
When this limit exists, $(Q^x)_{x\in I}$ defines the law of a diffusion, called the $Q$-process, that never reaches 0. See \cite{Ca09}
\end{definition}
Conditions for the existence and uniqueness of the $Q$-process have been studied in \cite{Co95, Ca09, CV14} while some examples of applications to ecology and genetics appear in \cite{L08, E14}.

One of the simplest ways of modeling a population living in one patch is assuming it is given by a geometric Brownian motion $(Y_t)_{t\geq 0}$, see \cite{E13, EHS14}. Then $(Y_t)_{t\geq 0}$ can be recovered as the solution to the following SDE
\begin{equation}\label{e:pop1}
\begin{split}
dY_t &= \mu Y_t\,dt + \sigma Y_t dW_t,\\
Y_0 &= y>0.\\
\end{split}
\end{equation}
where $(W_t)_{t \ge 0}$ is a standard Brownian motion; that is, there
is exponential growth with a rate that varies stochastically in time.  Here
$Y_t = y \exp\left(\sigma W_t + \left(\mu-\frac{\sigma^2}{2}\right) t\right)$, $t \ge 0$.
Starting from any $y\in (0,\infty)$
the process $(Y_t)_{t\geq 0}$ does not hit $0$: the population will not go extinct in finite time.
The long term behavior of \eqref{e:pop1} is determined by the {\em stochastic growth rate} $\mu-\frac{\sigma^2}{2}$.
\begin{itemize}
\item If $\mu-\frac{\sigma^2}{2}>0$ then $\lim_{t\rightarrow \infty}Y_t=\infty$, $\Pr^y$-almost surely for all $y\in(0,\infty)$.
\item If $\mu-\frac{\sigma^2}{2}<0$ then $\lim_{t\rightarrow \infty}Y_t=0$, $\Pr^y$-almost surely for all $y\in(0,\infty)$.
\item If $\mu-\frac{\sigma^2}{2}=0$ then $(Y_t)_{t\geq 0}$ is null-recurrent.
\end{itemize}

In the cases where the geometric Brownian motion $(Y_t)_{t\geq 0}$ is transient,
one type of conditioning is to consider the limit as
$t \to \infty$ of the process obtained by conditioning on the event $\{Y_t = a\}$ for some fixed $a > 0$.  It is not hard to see
by direct computation using transition densities that this conditional distribution does not depend on $\mu$,
that the limit exists, the limit does not depend on $a$, and the
limit is just the (unconditional) distribution of $Y$ with $\mu-\frac{\sigma^2}{2}=0$.

In general, if $(X_t)_{t\geq 0}$ has transition densities $p_t(x,y)$ with respect to some reference measure $m$, then
$(X_s)_{0 \le s \le t}$ conditional on the event $\{X_t = a\}$ is a time-inhomogeneous Markov process with
transition densities
\[
q_{r,s}^{(t)}(x,y) := \frac{p_{s-r}(x,y) p_{t-s}(y,a)}{p_t(x,a)}.
\]
Under appropriate conditions, the limit
\[
h_u(z) := \lim_{v \to \infty} \frac{p_{v+u}(z,a)}{p_v(a,a)}
\]
exists, is strictly positive, and $\bP^x[h_u(X_w)] = h_{u+w}(x)$ for all $x$ in
the state space and all $u \in \bR_+$; see \cite{CW05, Bor02}.  In such situations, the limit of the conditioned processes
is the {\em Doob $h$-transform} process, a time-homogeneous Markov process with transition densities
\[
q_t(x,y) := h(x)^{-1} p_t(x,y) h(y).
\]
If one adds competition for resources in \eqref{e:pop1} then the population $(\tilde Y_t)_{t\geq 0}$ is modeled by
\begin{equation}\label{e:pop2}
\begin{split}
d\tilde Y_t &= (\mu \tilde Y_t-\kappa \tilde Y_t^2)\,dt + \sigma \tilde Y_t dW_t,\\
\tilde Y_0 &= \tilde y>0.\\
\end{split}
\end{equation}
where $\kappa>0$ is the strength of intraspecific competition. One can show, see \cite{EHS14}, that almost surely the population $(\tilde Y_t)_{t\geq 0}$ does not go extinct in finite time and that when $\mu-\frac{\sigma^2}{2}<0$ one has
\[
\lim_{t\rightarrow \infty}\tilde Y_t=0
\]
$\Pr^{\tilde y}$-almost surely for all $\tilde y\in(0,\infty)$.

In the models defined by \eqref{e:pop1} and \eqref{e:pop2} with $\mu-\frac{\sigma^2}{2}<0$ the populations stay positive for all $t\geq 0$ and go extinct asymptotically as $t\rightarrow\infty$.
As a result we felt that it is interesting to study an analogue of the $Q$-process, where we condition that the population does not go extinct asymptotically.
\begin{example}\label{e:1}
Suppose $(X_t)_{t\geq 0}$ is a Brownian motion with negative drift $-\mu<0$. That is,
\[
X_t= W_t-\mu t, ~t\geq 0
\]
where $(W_t)_{t\geq 0}$ is a standard Brownian motion. It is well known that
\[
\lim_{t\rightarrow\infty} X_t = -\infty
\]
$\Pr^x$-almost surely for all $x\in\R$. If we condition  $(X_t)_{t\geq 0}$ on the event $\{X_T\in(a,\infty)\}$ and then let $T\rightarrow \infty$ and $a\rightarrow \infty$ such that $\frac{a}{T}\rightarrow c$ it is known that one gets a process $(Z_t)_{t\geq 0}$ that is Brownian motion with  drift $c$,
\[
Z_t=W_t+c t, ~t\geq 0.
\]
In other words, if we take Brownian motion with negative drift $-\mu$ and condition it `to go to $+\infty$' we get a process that is Brownian motion with positive drift $c>0$.
\end{example}

We show, in Theorem \ref{t:generator_diffusion}, that when we have a diffusion $(\hat X_t)_{t\geq 0}$ on $(\ell,\infty)$ such that
\[
\hat X_t\rightarrow \ell
\]
as $t\rightarrow \infty$ almost surely for any starting point $x\in (\ell,\infty)$ and we suitably `condition' $\hat X$ to go to $\infty$ as $t\rightarrow \infty$ what we get is the $h$-transform of $(\hat X_t)_{t\geq 0}$ with $h=s$ where $s$ is the scale function of $(\hat X_t)_{t\geq 0}$. The conditioning works as follows: we let $\zeta$ be an independent exponential with rate $\lambda$, condition $(X_t)_{t\geq 0}$ on $\{ \hat X_{\zeta-}\in(a,\infty)\}$, kill the process at $\zeta$, and then let $a\rightarrow \infty$ followed by $\lambda\rightarrow\infty$. Our result is similar to Proposition 3.2 from \cite{PR12}.

As examples we explore what this conditioning does when applied to the models \eqref{e:pop1} and \eqref{e:pop2}. We show that, as expected, the conditioning of a Brownian motion with negative drift $(W_t-\mu t)_{t\geq 0}$ gives us a Brownian motion with positive drift $(W_t+\mu t)_{t\geq 0}$. For the model \eqref{e:pop2} we show that the conditioning makes our process look, for large values, like
\[
d\bar Y_t=(\mu \bar Y_t + \kappa \bar Y_t^2)\,dt + \sigma \bar Y_t \,dW_t
\]
while for small values it behaves similar to

\[
dZ_t = ((\sigma^2-\mu) Z_t-\kappa Z_t^2)\,dt+\sigma Z_t \,dW_t.
\]

\section{Limits of diffusions conditioned to go to infinity at their terminal times}\label{s:infinity}

We denote by $C(I), C_b(I)$ the continuous and the continuous bounded functions on $I$ and by $\mathcal{B}(I)$ the Borel subsets of $I$.

A \textit{diffusion} is a strong Markov process with continuous paths. For the sake of completeness we present some basic facts about diffusions in Appendix A.

\begin{remark}
Suppose the diffusion $(X_t)_{t\geq 0}$ has null killing measure and scale and speed measures that
are absolutely continuous with respect to Lebesgue measure
\begin{itemize}
  \item  $m(dx)=m'(x) \, dx$.
  \item  $s(dx)=s'(x) \, dx$.
  \item  $k\equiv 0$.
\end{itemize}
If furthermore, one assumes $m'\in C(I)$ and $s'\in C^1(I)$ then one can show that the infinitesimal generator $\cG:\cD(\cG)\mapsto\cC_b(I)$ of $(X_t)_{t\geq 0}$ is a second order differential operator
\[
\cG f(x) = \frac{1}{2}\sigma^2(x)\partial_{xx}f(x) + b(x)\partial_x f(x)
\]
where
\[
m'(x) = 2\sigma^{-2}(x)e^{B(x)}, s'(x)=e^{-B(x)}
\]
with $B(x):=\int^x2\sigma^{-2}(y)b(y)\,dy$. The domain $\cD(\cG)$ consists of all functions in $\cC_b(I)$ such that $\cG f\in\cC_b(I)$ together with the appropriate boundary conditions.
\end{remark}

The following Theorem is a generalization of Example \ref{e:1}. A similar result appears in Proposition 3.2 of \cite{PR12}.
\begin{theorem}\label{t:generator_diffusion}
Let $(X_t)_{t\geq 0}$ be a regular one-dimensional diffusion on $(\ell,\infty)$ with semigroup $(P_t)_{t\geq 0}$, scale function $s$ and killing measure $k\equiv 0$.
We make the following assumptions:
\begin{itemize}
\item The boundary points $\{\ell,\infty\}$ are inaccessible.
\item $(X_t)_{t\geq 0}$ is transient and $\lim_{t\rightarrow \infty}X_t=\ell$, $\Pr^x$- almost surely for all $x\in(\ell,\infty)$.
\item  $\zeta$ is an independent exponential random variable with rate $\lambda$.
\end{itemize}
Kill $(X_t)_{t\geq 0}$ at $\zeta$ and condition on $\{X_{\zeta-}\in(a,\infty)\}$. The limit as $a\rightarrow \infty$ followed by the limit as $\lambda\downarrow 0$ of the conditioned killed process is a diffusion $(Z_t)_{t\geq 0}$ with semigroup $(Q_t)_{t\geq 0}$ given by
\[
Q_t f = \frac{1}{s} P_t(sf).
\]

Furthermore if the generator of $(X_t)_{t\geq 0}$ acts on functions with compact support in $C^2((\ell,\infty))$ as 

\[
\cG f(x):=\frac{1}{2} \sigma^2(x) \frac{d^2f}{dx^2}(x) + b(x) \frac{df}{dx}(x)
\]
with $b, \sigma\in C(I)$ and $\sigma^2(x)>0$ for all $x\in (\ell,\infty)$ then the generator of $(Z_t)_{t\geq 0}$ acts on functions with compact support in $C^2((\ell,\infty))$ as 

\begin{equation}\label{e:generator}
\cG^s f(x):=\frac{1}{2} \sigma^2(x) \frac{d^2f}{dx^2}(x)
+
\left[b(x)
+
\frac{\sigma^2(x)\exp\left(-\int_w^x 2 \frac{b(z)}{\sigma^2(z)} \, dz\right)}
{\int_\ell^x \exp\left(-\int_w^y 2 \frac{b(z)}{\sigma^2(z)} \, dz\right) \, dy}
\right]
\frac{df}{dx}(x),
\end{equation}
where $w\in I$ is an arbitrary reference point. 
\end{theorem}
\begin{remark}
We note that we get the same result in Theorem \ref{t:generator_diffusion} if we take limits in the other order, namely if we let $\lambda\downarrow 0$ followed by $a\rightarrow \infty$.
\end{remark}

\begin{proof}
Suppose first that  $(X_t)_{t\geq 0}$ is a one-dimensional diffusion in natural scale.
By Theorem V.50.7 from \cite{RW00}, the Green function (or the density of the resolvent of $(X_t)_{t\geq 0}$ against the speed measure) is given by
\begin{equation}\label{e:resolvent_density}
r_\lambda(x,y) =
\begin{cases} c_\lambda \psi_\lambda^+(x) \psi_\lambda^-(y) &\mbox{if } x \le y\\
  c_\lambda \psi_\lambda^-(x) \psi_\lambda^+(y)  & \mbox{if } y \le x,
\end{cases}
\end{equation}
for a constant $c_\lambda$ and
certain functions $\psi_\lambda^+$ and $\psi_\lambda^-$.  Suppose that
$A = (a,\infty)$ for some $a$ and that $\zeta$ is an independent exponential random variable with rate $\lambda>0$. Using the technical result, Corollary \ref{t:Mar} from Appendix B,  we know that $(X_t)_{t\geq 0}$ killed at $\zeta$ and conditioned on $\{X_{\zeta-} \in A\}$ is a diffusion with semigroup $(\bar P_t)_{t \ge 0}$
\begin{equation}\label{e:semigroup_diff}
\begin{split}
\bar P_t g(x)
& =
\frac{1}{\Pr^x\{X_{\zeta-} \in A\}}
\E^x[g(X_t) \ind\{\zeta > t\} \Pr^{X_t}\{X_{\zeta-} \in A\}] \\
& =
\frac{1}{\int_a^\infty r_\lambda(x,y) \, m(dy)}
\E^x[g(X_t)  \int_a^\infty r_\lambda(X_t, y) \, m(dy)]. \\
\end{split}
\end{equation}
Assume that $x < a$. Then, using \eqref{e:semigroup_diff} and \eqref{e:resolvent_density} we have
\begin{eqnarray*}
\bar P_t g(x)
& = &
\frac{1}{c_\lambda \psi_\lambda^+(x) \int_a^\infty \psi_\lambda^-(y) \, m(dy)}\times \\
&~&~\times~
\E^x\left[g(X_t)  \left(\int_a^{a \vee X_t} r_\lambda(X_t, y) \, m(dy)
+ c_\lambda \psi_\lambda^+(X_t) \int_{a \vee X_t}  \psi_\lambda^-(y) \, m(dy)\right)\right].
\end{eqnarray*}
Then
\begin{equation*}
\lim_{a \to \infty}\bar P_t g(x) = \frac{1}{\psi_\lambda^+(x)}
\E^x[g(X_t) \psi_\lambda^+(X_t)]
\end{equation*}
which yields

\begin{equation}\label{e:psi_hitting}
\lim_{\lambda \downarrow 0}\lim_{a \to \infty}\bar P_t g(x) =  \frac{1}{\psi_0^+(x)}
\E^x[g(X_t) \psi_0^+(X_t)].
\end{equation}
We know that
\[
\frac{\psi_0^+(y)}{\psi_0^+(x)}
=
\frac{\Pr^y\{T_z < \infty\}}{\Pr^w\{T_z < \infty\}}
\bigg /
\frac{\Pr^x\{T_z < \infty\}}{\Pr^w\{T_z < \infty\}}
\]
where $T_z$ is the hitting time of $z$.
We have assumed that we
are working in natural scale, but we see that this last quantity does not
depend on that assumption.  By assumption we have a process
on the interval $(\ell,\infty)$ such that the endpoints are inaccessible and
$X_t \to \ell$ as $t \to \infty$.  Then, by the definition of the scale
function $s$,
\begin{equation}\label{e:hitting}
\Pr^y\{T_z < \infty\}
= \lim_{u \downarrow \ell} \frac{s(y) - s(u)}{s(z) - s(u)},
\end{equation}
with similar formulas for the other hitting probabilities.
So, for the original process,
\begin{equation}\label{e:scale_H}
\frac{\Pr^y\{T_z < \infty\}}{\Pr^w\{T_z < \infty\}}
\bigg /
\frac{\Pr^x\{T_z < \infty\}}{\Pr^w\{T_z < \infty\}}
=
\lim_{u \downarrow \ell} \frac{s(y) - s(u)}{s(x) - s(u)}.
\end{equation}

The assumption that $(X_t)_{t\geq 0}$ wanders off to the left boundary point $\ell$
implies that $\lim_{u \downarrow \ell} s(u) \ne -\infty$ and so
we can assume that the scale function is chosen so that this limit is $0$. Equation \eqref{e:scale_H} becomes

\begin{equation}\label{e:scale_HH}
\frac{\Pr^y\{T_z < \infty\}}{\Pr^w\{T_z < \infty\}}
\bigg /
\frac{\Pr^x\{T_z < \infty\}}{\Pr^w\{T_z < \infty\}}
=
\frac{s(y)}{s(x)}.
\end{equation}

Our limit semigroup is therefore
\begin{equation*}
\begin{split}
Q_t g(x) &= \frac{1}{\psi_0^+(x)} \E^x[g(X_t) \psi_0^+(X_t)]\\
&=\frac{1}{s(x)} \E^x[g(X_t) s(X_t)].\\
\end{split}
\end{equation*}
Note that this is just the $h$-transform with $h:=s$ of the semigroup of $(X_t)_{t\geq 0}$.
If  $(X_t)_{t\geq 0}$ has generator $\cG$ acting on functions with compact support in $C^2((\ell,\infty))$ as
\[
\cG f(x)=\frac{1}{2} \sigma^2(x) \frac{d^2f}{dx^2}(x) + b(x) \frac{df}{dx}(x)
\]
then the scale function will be
\[
s(x) = \int_\ell^x \exp\left(-\int_w^y 2 \frac{b(z)}{\sigma^2(z)} \, dz\right) \, dy,
\]
while the speed measure will have density
\[
m'(x) = 2\sigma^{-2}(x)\exp\left(-\int_w^x 2 \frac{b(z)}{\sigma^2(z)} \, dz\right)
\]
where $w$ is an arbitrary reference point.
Note that
\[
\cG s(x) =0
\]
for all $x\in (\ell,\infty)$.

According to Theorem \ref{t:h_diffusion_chars} from Appendix B we see that the generator associated with
the semigroup $(Q_t)_{t \ge 0}$ acts on functions with compact support in $C^2((\ell,\infty))$ as
\begin{equation*}
\begin{split}
\cG^s f(x)&=\frac{1}{s^2(y)m'(y)}\left(\frac{s^2(y)}{s'(y)}f'(y)\right)'\\
&=\frac{1}{2}\sigma^2(x)\frac{d^2f}{dx^2} + \left(b(x)+\sigma^2(x)\frac{s'(x)}{s(x)}\right)\frac{df}{dx}\\
&=\frac{1}{2} \sigma^2(x) \frac{d^2f}{dx^2}+\left[b(x) +\frac{\sigma^2(x)\exp\left(-\int_w^x 2 \frac{b(z)}{\sigma^2(z)} \, dz\right)}
{\int_\ell^x \exp\left(-\int_w^y 2 \frac{b(z)}{\sigma^2(z)} \, dz\right) \, dy}
\right]
\frac{df}{dx}.
\end{split}
\end{equation*}
\end{proof}
\begin{remark} If the diffusion $(X_t)_{t\geq 0}$ is a strictly positive local martingale then one can see that on $I=(0,\infty)$ the following conditions hold
\begin{itemize}
\item The boundary points $0,\infty$ are inaccessible.
\item $\lim_{t\to \infty} X_t =0$, $\Pr^x$ - almost surely for all $x\in I$.
\end{itemize}
The last conditions holds because of the following argument: 

A strictly positive continuous local martingale $X$ is a supermartingale since $X$ is bounded below by $0$. Also, if $X_t^-:=\max(-X_t,0)$ then $X_t^-\equiv 0$ almost surely so
\[
\sup_{t>0} \E^x[X_t^-] = 0<\infty
\]
By Doob's first martingale convergence theorem we get that
\be\label{e:conv}
\Pr^x\{X_\infty:= \lim_{t\to \infty} X_t~ \text{exists and}~0\leq X_\infty<\infty\}=1
\ee
for all $x\in I$. Then since $\Pr^x(T_a\wedge T_b<\infty)=1$ for any $a\leq x\leq b$ we see that
\[
\liminf_{t\to \infty}X_t \leq a
\]
or
\[
\limsup_{t\to \infty}X_t \geq b.
\]
This fact combined with \eqref{e:conv} yields that
\[
\Pr^x\{X_\infty = 0\}=1
\]
for all $x\in I$.

Note that we can get diffusions that are strictly positive martingales using the following proposition.
\begin{proposition}
Suppose $(Y_t)_{t\geq 0}$ is the solution to the SDE

\begin{equation}\label{e:local_MG}
\begin{split}
dY_t &= \sigma(Y_t)Y_tdW_t,\\
Y_0 &= y>0.\\
\end{split}
\end{equation}
where $\sigma(x)>0$ for all $x\in (0,\infty)$ and $\sigma^{-2}$ is locally integrable on $(0,\infty)$. Then, solutions $Y$ to \eqref{e:local_MG} do not hit zero almost surely for any starting point $y\in (0,\infty)$ if and only if
\[
\int_0^K x^{-1}\sigma^{-2}(x)\,dx = \infty
\]
for some $K>0$.

\end{proposition}

\end{remark}
\begin{remark}\label{r:sde}
Suppose the functions $\sigma, \mu\in C((\ell,r))$ satisfy
\begin{itemize}
\item $\sigma(x)>0$ for all $x\in (\ell, r)$.
\item $\frac{1}{\sigma^2(\cdot)}, \frac{b(\cdot)}{\sigma^2(\cdot)}$ are locally integrable on $(\ell,r)$.
\end{itemize}
The stochastic differential equation
\be
\begin{split}\label{e:SDE}
d Y_t &= \sigma(Y_t) dW_t + \mu(Y_t) dt\\
Y_0 &= y\in (\ell,r)
\end{split}
\ee
has a solution for each $y\in (\ell,r)$ that does not explode and is unique in law. Then $(Y_t)_{t\geq 0}$
is a regular diffusion with scale function density and speed measure density given by
\be\label{e_SDE_scale_speed_killing}
\begin{split}
m'(x)&=\frac{2}{\sigma^2(x)}\exp\left(\int^x \frac{2}{\sigma^2(y)}\mu(y)\,dy\right),\\
s'(x) &= \exp\left(-\int^x \frac{2}{\sigma^2(y)}\mu(y)\,dy\right).
\end{split}
\ee
\end{remark}

By Remark \ref{r:sde} and Theorem \ref{t:generator_diffusion} we have the following Corollary.

\begin{corollary}
Let $(X_t)_{t\geq 0}$ be the solution to the one dimensional stochastic differential equation
\begin{equation*}
\begin{split}
d X_t &= \sigma(X_t) dW_t + b(X_t) dt\\
X_0 &= x\in I
\end{split}
\end{equation*}
on $I=(\ell,\infty)$ with $\ell, \infty$ inaccessible boundary points and such that
\[
\Pr^x\left\{\lim_{t\rightarrow \infty} X_t = \ell\right\}=1
\]
 for all $x\in (\ell,\infty)$. Assume that 
 \begin{itemize}
\item $\sigma^2(x)>0$ for all $x\in I$.
\item $\frac{1}{\sigma^2(\cdot)}, \frac{b(\cdot)}{\sigma^2(\cdot)}$ are locally integrable on $I$.
\item $\zeta$ is an independent exponential with rate $\lambda$.
\end{itemize}
If we condition $(X_t)_{t\geq 0}$ on $\{X_{\zeta-}\in (a,\infty)\}$ for $a\in(\ell,\infty)$, kill the process at $\zeta$, and let $a\rightarrow \infty$ followed by $\lambda\downarrow 0$ we get a diffusion $(Z_t)_{t\geq 0}$ that can be represented as the solution to the SDE
\begin{equation*}
d Z_t = \left[b(Z_t)
+
\frac{\sigma^2(Z_t)\exp\left(-\int_w^{Z_t} 2 \frac{b(z)}{\sigma^2(z)} \, dz\right)}
{\int_\ell^{Z_t} \exp\left(-\int_w^y 2 \frac{b(z)}{\sigma^2(z)} \, dz\right) \, dy}
\right] + \sigma(Z_t)\,dW_t.
\end{equation*}
\end{corollary}

\subsection{Examples}\label{s:examples}

\subsubsection{Brownian motion with negative drift} We do this type of construction with Brownian motion that has a negative
drift $-\mu<0$
\begin{equation*}
dX_t = -\mu \,dt + dW_t, t\geq 0.
\end{equation*}
Note that in this case $\lim_{t\rightarrow \infty}X_t = -\infty$ $\Pr^x$-almost surely for all $x\in (-\infty, \infty)$ and $\ell=-\infty$. Since $b(x)=-\mu$ and $\sigma(x)=1$ the scale function will be given by

\begin{eqnarray*}
s(x)&=&\int_{-\infty}^x \exp\left(-\int_w^y 2 \frac{b(z)}{\sigma^2(z)} \, dz\right) \, dy\\
&=& \frac{1}{2\mu}e^{2\mu(y-w)} .
\end{eqnarray*}
As a result
\begin{eqnarray*}
\frac{s'(x)}{s(x)} &=& 2\mu
\end{eqnarray*}
and the drift $\tilde b$ of the conditioned limiting process $(Z_t)_{t\geq 0}$ is
\begin{eqnarray*}
\tilde b(x) &=&\left[b(x)+ \frac{\sigma^2(x)\exp\left(-\int_w^x 2 \frac{b(z)}{\sigma^2(z)} \, dz\right)} {\int_\ell^x \exp\left(-\int_w^y 2 \frac{b(z)}{\sigma^2(z)} \, dz\right) \, dy}\right] \\
&=& (-\mu + 2\mu)\\
&=& \mu.
\end{eqnarray*}
The limiting process $(Z_t)_{t\geq 0}$ is a  Brownian motion with positive drift $\mu$
\[
dZ_t =  \mu\, dt  + dW_t.
\]

\subsubsection{An example from population dynamics}
Next, let us look what happens to the diffusion
\begin{equation*}
dX_t = X_t(\mu-\kappa X_t)\,dt + \sigma X_t \,dW_t.
\end{equation*}
This SDE models the total population abundance of a species living in one patch. The intuitive meaning of the coefficients is the following
\begin{itemize}
\item $\mu$ is the intrinsic growth rate of the population in the absence of stochasticity,
\item $\kappa$ is the strength of intraspecific competition,
\item $\sigma^2$ is the infinitesimal variance parameter of the stochastic growth rate.
\end{itemize}
We consider the case when $\mu-\frac{\sigma^2}{2}<0$. When this condition is satisfied one knows that the SDE has a strong solution which does not explode, satisfies $X_t>0$ for all $t\geq0$ and goes asymptotically to zero
\[\lim_{t\rightarrow \infty}X_t=0
\]
$\Pr^x$-almost surely for all $x\in\R_+$. See for example \cite{EHS14}.
Using the construction above with $\ell=0$, $b(z)=\mu z-\kappa z^2$ and $\sigma(z)=\sigma z$ we get that the scale function is
\begin{eqnarray*}
s(x)&=&\int_0^x \exp\left(-\int_w^y 2 \frac{b(z)}{\sigma^2(z)} \, dz\right) \, dy\\
&=& \int_0^x \left(\frac{y}{w}\right)^{-\frac{2\mu}{\sigma^2}} \exp\left(\frac{2\kappa}{\sigma^2}(y-w)\right)\,dy.
\end{eqnarray*}
The logarithmic derivative of $s(x)$ will be given by
\begin{eqnarray*}
\frac{s'(x)}{s(x)} &=& \frac{x^{-\frac{2\mu}{\sigma^2}} \exp\left(\frac{2\kappa}{\sigma^2}x \right)}{\int_0^x y^{-\frac{2\mu}{\sigma^2}} \exp\left(\frac{2\kappa}{\sigma^2}y \right)\,dy}.
\end{eqnarray*}
As a result the drift of the conditioned diffusion $(Z_t)_{t\geq 0}$ is
\begin{eqnarray*}
\tilde b(x) &=&\left[b(x)+ \frac{\sigma^2(x)\exp\left(-\int_w^x 2 \frac{b(z)}{\sigma^2(z)} \, dz\right)} {\int_\ell^x \exp\left(-\int_w^y 2 \frac{b(z)}{\sigma^2(z)} \, dz\right) \, dy}\right] \\
&=& \mu x-\kappa x^2+\sigma^2x^2  \frac{x^{-\frac{2\mu}{\sigma^2}} \exp\left(\frac{2\kappa}{\sigma^2}x \right)}{\int_0^x y^{-\frac{2\mu}{\sigma^2}} \exp\left(\frac{2\kappa}{\sigma^2}y \right)\,dy}\\
&=& \mu x-\kappa x^2 + \frac{\sigma^2 x^{-\frac{2\mu}{\sigma^2}+2} \exp\left(\frac{2\kappa}{\sigma^2}x \right)}{\int_0^x y^{-\frac{2\mu}{\sigma^2}} \exp\left(\frac{2\kappa}{\sigma^2}y \right)\,dy}.
\end{eqnarray*}
Let us study the asymptotics as $x\rightarrow \infty$ of $\tilde b(x)$. First look at
\[
\int_0^x y^{-\frac{2\mu}{\sigma^2}}\exp\left(\frac{2\kappa}{\sigma^2}(y-x)\right)\,dy.
\]
Do the change of variables
\[
z=-\frac{2\kappa}{\sigma^2}(y-x).
\]
As a result
\[
y=x-\frac{\sigma^2}{2\kappa}z
\]
and
\[
dy = -\frac{\sigma^2}{2\kappa}\,dz
\]
so we can write
\begin{equation*}
\begin{split}
\int_0^x y^{-\frac{2\mu}{\sigma^2}}\exp\left(\frac{2\kappa}{\sigma^2}(y-x)\right)\,dy &=\int_0^{\frac{2\kappa x}{\sigma^2}} \left(x-\frac{\sigma^2z}{2\kappa}\right)^{-\frac{2\mu}{\sigma^2}}\exp(-z)\frac{\sigma^2}{2\kappa}\,dz.\\
\end{split}
\end{equation*}
Thus
\begin{equation*}
\begin{split}
\frac{\sigma^2 x^{-\frac{2\mu}{\sigma^2}+2} \exp\left(\frac{2\kappa}{\sigma^2}x \right)}{\int_0^x y^{-\frac{2\mu}{\sigma^2}} \exp\left(\frac{2\kappa}{\sigma^2}y \right)\,dy}&=\frac{\sigma^2 x^2}{\int_0^{\frac{2\kappa x}{\sigma^2}} \left(1-\frac{\sigma^2}{2\kappa}\frac{z}{x}\right)^{-\frac{2\mu}{\sigma^2}}\exp(-z)\frac{\sigma^2}{2\kappa}\,dz}\\
&\approx \frac{2\kappa x^2}{\int_0^\infty e^{-z}\,dz} \\
&= 2\kappa x^2
\end{split}
\end{equation*}
as $x\rightarrow\infty$ which shows that
\[
\tilde b(x) \approx \mu x-\kappa x^2 + 2\kappa x^2 = \mu x+\kappa x^2
\]
as $x\rightarrow \infty$.

This implies that the limiting diffusion $(Z_t)_{t\geq 0}$ looks, for large values of $(Z_t)_{t\geq 0}$, like

\[
d\bar Y_t = (\mu \bar Y_t+\kappa \bar Y_t^2)\,dt+\sigma \bar Y_t \,dW_t.
\]
\begin{remark}\label{r: bar Y}
One can show that the process $(\bar Y_t)_{t\geq 0}$ explodes in finite time.
\end{remark}
Let us next study the asymptotics as $x\rightarrow 0$ of $\tilde b(x)$.

\begin{eqnarray*}
\tilde b(x)&=& \mu x-\kappa x^2 + \frac{\sigma^2 x^{-\frac{2\mu}{\sigma^2}+2} \exp\left(\frac{2\kappa}{\sigma^2}x \right)}{\int_0^x y^{-\frac{2\mu}{\sigma^2}} \exp\left(\frac{2\kappa}{\sigma^2}y \right)\,dy}\\
&\approx& \mu x-\kappa x^2 + \frac{\sigma^2 x^{-\frac{2\mu}{\sigma^2}+2} }{\int_0^x y^{-\frac{2\mu}{\sigma^2}} \,dy}\\
&=& \mu x-\kappa x^2 + \frac{\sigma^2 x^{-\frac{2\mu}{\sigma^2}+2} }{\frac{x^{-\frac{2\mu}{\sigma^2}+1}}{-\frac{2\mu}{\sigma^2}+1}}\\
&=& \mu x-\kappa x^2 + x(-2\mu+\sigma^2)
\end{eqnarray*}

The conditioned diffusion $(Z_t)_{t\geq 0}$ looks, for small values of $(Z_t)_{t\geq 0}$, like

\[
d\hat Z_t = ((\sigma^2-\mu) ]\hat Z_t-\kappa ]\hat Z_t^2)\,dt+\sigma \hat Z_t \,dW_t.
\]
\begin{remark}
Note that
\[
(\sigma^2-\mu)-\frac{\sigma^2}{2} = \frac{\sigma^2}{2}-\mu>0
\]
so that the process $(\hat Z_t)_{t\geq 0}$ is not going to go to zero. This together with Remark \ref{r: bar Y} show that the process $(Z_t)_{t\geq 0}$ explodes in finite time.
\end{remark}

\section*{Appendix A}
We follow \cite{Bor02} for some basic facts about diffusions.
In this paper we only consider \textit{regular diffusions}; that is, diffusions such that for all $x,y\in I$
\[
\Pr^x\{T_y<\infty\}>0
\]
where $T_y:=\inf\{t:X_t=y\}$ -- any state $y$ can be reached in finite time with positive probability from any state $x$.

The diffusion $(X_t)_{t\geq 0}$ determines  three basic Borel measures on the state space $I$:
a {\em scale measure} $s$, a {\em speed measure} $m$,  and a {\em killing measure} $k$ (see \cite{IM96}). From now on we will consider that there is no killing, that is $k\equiv 0$.
It turns out to be convenient not to specify these objects absolutely but only up to a constant.
If $(s^*, m^*)$ and $(s^{**}, m^{**})$ are two pairs of these objects, then
$s^{**} = c s^*$ for some strictly positive constant $c$, in which case $m^{**} = c^{-1} m^*$.  The scale measure $s$ is  diffuse.  Both the scale measure and the speed measure
have full support and assign
finite mass to intervals of the form $(y,z)$, where $\ell < y < z < r$.
If $(P^X_t)_{t \ge 0}$ is the transition semigroup of $(X_t)_{t\geq 0}$, then there exists a density $p$ that is
strictly positive, jointly continuous in all variables, and symmetric such that
\[
P^X_t(x,A)=\int_A p(t;x,y) \, m(dy), \quad \text{$x\in I$, $t>0$, and $A\in \mathcal{B}(I)$}.
\]

With a standard abuse of notation, as well as using $s$
to denote the scale measure we write $s$ for any {\em scale function} such that
\[
s(z) - s(y) =\int_y^z \, s(dx).
\]
For $\alpha>0$ the \textit{Green function} $r_\alpha(x,y)$ is given by
\begin{equation*}
r_\alpha(x,y):= \int_0^\infty e^{-\alpha t} p(t;x,y)\,dt,
 \end{equation*}
where $p(t;x,y)$ is the transition density with respect to the speed measure $m$. Set
\[
D^+_{s}f(x) = \lim_{\eta\downarrow x}\frac{f(\eta)-f(x)}{s(\eta)-s(x)},
\]
and
\[
D^-_{s}f(x) = \lim_{\eta\uparrow x}\frac{f(\eta)-f(x)}{s(\eta)-s(x)}
\]
for a function $f:(a,b) \to \bR$.

The diffusion $(X_t)_{t\geq 0}$ determines and in turn is determined by its infinitesimal generator.
The infinitesimal generator is specified by the scale, speed and killing measures
and by {\em boundary conditions} on functions in the domain.

\begin{definition}\label{d:generator}
The \textit{(weak) infinitesimal generator} of $(X_t)_{t\geq 0}$ is the operator $\cG$ defined by
\[
\cG^{}f:=\lim_{t\downarrow 0}\frac{P_t f-f}{t}
\]
applied to $f\in \cC_b(I)$ for which the limit exists pointwise, is in $\cC_b(I)$, and
\[
\sup_{t>0}\left\|\frac{P_t f-f}{t}\right\|<\infty.
\]
We denote by $\cD(\cG^{})$ the set of such functions. One can show that $\cD(\cG)$ is characterized by saying that $f\in \cC_b(I)$
belongs to $\cD(\cG)$ if $D_{s}^-f$ and $D_{s}^+f$ exist
and there exists a function $g\in \cC_b(I)$ such that for all $\ell<a<b<r$,
\begin{itemize}
  \item [(a)]
  \[
  \int_{[a,b)}g(x) \, m(dx) = D_{s}^- f(b) - D_{s}^- f(a) - \int_{[a,b)}f(x) \, k(dx).
  \]
   \item [(b)]
  \[
  \int_{(a,b]}g(x) \, m(dx) = D_{s}^+ f(b) - D_{s}^+ f(a) - \int_{(a,b]}f(x) \, k(dx).
  \]

 \end{itemize}
 together with boundary conditions. See \cite{Bor02} for more details.
 \end{definition}

\section*{Appendix B}
Let $(X_t, \Omega, \mathcal{F}, \Pr^x, \theta_t,(\mathcal{F}_t))$ be a strong Markov process with state space a locally compact metric space $E$. We let $\Omega$ be the space of functions $\omega:\R_+\mapsto E_\partial$ which are right continuous, admit an almost surely finite terminal time $\zeta<\infty$, and have left limits. We can define $(X_t)_{t\geq 0}$ on this probability space by $X_t(\omega)=\omega(t)$, $\omega\in \Omega$. This spaces comes equipped with the shift operator $\theta_t$:
\[
\theta_t\omega(s) = \omega(s+t).
\]
We let $\mathcal{F}_t^0$ be the natural filtration on
$\Omega$: $\mathcal{F}_t^0=\sigma\{X_s,0\leq s\leq t\}$.
Set $\mathcal{F}^0=\bigcup_{t}\mathcal{F}_t^0$ and for an initial law $\mu$ let $\mathcal{F}^\mu$ denote the completion of $\mathcal{F}^0$ relative to $\Pr^\mu$ and let $\mathcal{N}^\mu$ denote the $\Pr^\mu$-null sets in $\mathcal{F}^\mu$.

Define then
\begin{itemize}
\item $\mathcal{F}:= \bigcap \left\{\mathcal{F}^\mu: \mu ~\text{is an initial law on}~E\right\}$.
\item $\mathcal{N} :=  \bigcap \left\{\mathcal{N}^\mu: \mu ~\text{is an initial law on}~E\right\}$.
\item $\mathcal{F}_t^\mu:= \mathcal{F}_t\vee \mathcal{N}^\mu$.
\item $\mathcal{F}_t := \bigcap \left\{\mathcal{F}_t^\mu: \mu~\text{is an initial law on}~E\right\}$.
\end{itemize}
The process $X$ will be described by the probability family $(\Pr^x)_{x\in E}$ for which
\[
\Pr^x\{X_0=x\}=1
\]
for all $x\in E$.
We assumed $X$ has an almost surely finite terminal time $\zeta<\infty$. This means that $\zeta$ has the property
\begin{equation*}
\zeta = s + \zeta \circ \theta_s,
\end{equation*}
on the event $\{\zeta > s\}$.  In other words, if the process has not died by time $s$, then the decision about when to die comes from looking at the future piece of path as though we are starting at time zero.

We call a function $f:E\rightarrow \R_+\cup\{\infty\}$ \textit{excessive} if the following two conditions are satisfied
\begin{itemize}
\item[(1)]
\[
 \E^x f(X_t)\leq f(x)
\]
for all $t\geq 0$ and $x\in E$.
\item[(2)]
\[
\lim_{t\downarrow 0} \E^x f(X_t) = f(x)
\]
for all $x\in E$.
\end{itemize}
\begin{theorem}\label{t:conditioning}
Let $\bar  f:E\rightarrow \R_+$ be excessive. The operators $(P_t)_{t \geq 0}$ defined as
\[
P_t g(x) = \frac{1}{\bar f(x)} \Pr^x[1\{\zeta>t\}g(X_t)\bar f(X_t)]
\]
define a submarkovian semigroup for a family of probability measures $(\Qr^x)_{x\in I}$ on $\Omega$. The process $X$ is strong (sub)Markov under$(\Qr^x)_{x\in I}$.
\end{theorem}
\begin{proof}
Since $\bar f$ is excessive we can use Theorem 11.9 from page 325 in \cite{CW05} to say that under $(\Qr^x)_{x\in I}$ the process $X$ is a right-continuous killed strong Markov process which has left limits except possibly at its death time.
\end{proof}
The next result shows that the process $X$ conditioned to be in a set $A\subset E$ right before the terminal time $\zeta$ is strong Markov.
\begin{corollary}\label{t:Mar}
Assume that $\zeta$ is an independent exponential with rate $\lambda>0$ and that $A\subset E$ is such that
\[
\Pr^x\{X_{\zeta-}\in A\}>0
\]
for all $x\in E$. Then under the probability family defined by
\begin{equation}\label{e:Qx}
\Qr^x(B) = \Pr^x\{B~|~X_{\zeta-}\in A\}
\end{equation}
$X$ is a strong Markov process with transition semigroup
\begin{equation}\label{e:kernel}
\bar P_t g(x) := \frac{1}{\Pr^x\{X_{\zeta-} \in A\}}
\E^x[g(X_t) \ind\{\zeta > t\} \Pr^{X_t}\{X_{\zeta-} \in A\}].
\end{equation}

\end{corollary}
\begin{proof}
Let
\[
\bar f(x):=\Pr^x\{X_{\zeta-}\in A\}
\]

and note that
\[
\Pr^x\{X_{\zeta-}\in A\} = U_\lambda\ind_A(x) := \int_0^\infty e^{-\lambda t}\Pr^x\ind_A(X_t)\,dt
\]
where $U_\lambda$ is the $\lambda$-resolvent of $X$. Since $\ind_A$ is a bounded positive function we can apply proposition 2 from page 46 of \cite{CW05} to conclude that $\bar f$ is $\lambda$-excessive for $X$. As a result $\bar f$ is excessive for $X$ killed at the terminal time $\zeta$. The result now follows by applying Theorem \ref{t:conditioning}.
\end{proof}
The following result tells us how the characteristics of a diffusion change under an $h$-transform.
\begin{theorem}
\label{t:h_diffusion_chars}
Let $(X_t)_{t\geq 0}$ be a regular, transient diffusion living on $I$ with null killing measure,
speed measure $m$ and scale function $s$.
Suppose that $h$ is a strictly positive excessive function such that $h(x_0) = 1$ for some $x_0$ in the state space and that the boundary points of $I$ are inaccessible.  The Doob $h$-transform of $(X_t)_{t\geq 0}$ is
a regular diffusion $(X_t^h)_{t\geq 0}$ with the following characteristics:
 \begin{itemize}
 \item Scale measure
  \begin{equation}\label{e:h_scale}
 s^h(dy) = h^{-2}(y) \, s(dy).
\end{equation}
\item Speed measure
 \begin{equation}\label{e:h_speed}
 m^h(dy) = h^2(y) \, m(dy).
 \end{equation}
\item If $m'\in C(I)$ and $h, s'\in C^1(I)$ the generator of $(X_t^h)_{t\geq 0}$ acts on functions  with compact support in  $C^2((\ell,\infty))$ as
\[
 \cG^s f(x)=\frac{1}{h^2(y)m'(y)}\left(\frac{h^2(y)}{h'(y)}f'(y)\right)'.
\]
\end{itemize}
\end{theorem}

\begin{proof}
For a proof of this fact see \cite{EH15}. This result also seems to be known in the folk-lore but we were not able to find a proof for the general result.
\end{proof}

{\bf Acknowledgments.}  The author thanks Johannes Ruf, Steve Evans and Alison Etheridge for helpful discussions.

\bibliographystyle{amsalpha}
\bibliography{asymptotic}

\end{document}